\documentclass[11pt]{article}
\usepackage{amstext,amsmath,amssymb,amsthm}
\usepackage{latexsym}
\usepackage{exscale}
\usepackage{color}

\numberwithin{equation}{section}

\newcommand{\R}{{ \Bbb  R  }}

\newcommand{\B}{{\mathcal B}}
\newcommand{\EE}{\end{equation}}
\renewcommand{\t}{\tilde}
\renewcommand{\medskip}{\vskip .5 cm}
\newtheorem{Thm}{Theorem}[section]
\newtheorem{Lemma}[Thm]{Lemma}
\newtheorem{Cor}[Thm]{Corollary}
\newtheorem{Prop}[Thm]{Proposition}

\begin{document}

\centerline{\textbf{ \Large Parseval frames with $n+1$ vectors in $\R^n$} } \vskip.2cm

\medskip
\centerline{Laura De Carli and Zhongyuan Hu}

\addtolength{\itemsep}{0.3cm}
\bigskip
\begin{abstract}
We prove a uniqueness theorem for  triangular  Parseval frame  with $n+1$ vectors in $\R^n$.  We also provide a characterization of  unit-norm frames  that  can be scaled to a Parseval frame.
 \end{abstract}
 
 \centerline{Mathematical subject classification: 42C15, 46C99}
 
\section{Introduction}

Let $\B=\{v_1, ...,\,  v_N\}$ be a set of vectors in $\R^n$. We say that $ \B$ is  a {\it frame} if it   contains  a basis of $\R^n$, or equivalently, if     there exist constants $A,\ B>0$  
for which $A ||v||^2\leq \sum_{j=1}^N <v, v_j>^2\leq B||v||^2$ for every $v\in\R^n$.  Here and throughout the  paper, $<\ ,>$  and $||\ ||  $ are the  usual scalar product and norm  in $\R^n$. In general $A<B$, but we say that a frame is {\it tight}  if $A=B$, and  is {\it Parseval} if $A=B=1$. 

Parseval frames are    nontrivial generalizations of orthonormal bases. Vectors in a Parseval frame are not necessarily orthogonal or   linearly independent,  and  do not necessarily have the same length, but the {\it Parseval identities} $v= \sum_{j=1}^N <v, v_j> v_j$ and    $||v||^2= \sum_{j=1}^N <v, v_j>^2$
 still hold. In the applications, frames are more useful than bases because  they  are resilient against the corruptions of additive noise and quantization,
while providing numerically stable reconstructions (\cite{C}, \cite{D}, \cite{GVT}).
Appropriate
frame decomposition  may reveal “hidden” signal characteristics, and
have been employed as detection devices. Specific types of finite tight frames have been studied to solve problems in
information theory. The references are too many to cite, but see \cite{CFKLT},   the recent book \cite{C}  and the references cited there.
 
In recent years,
several inquiries about tight frames have been raised. 
In particular:  how  to characterize    Parseval  frames with $N$ elements in $\R^n$ (or {\it Parseval N frames}),
and   whether  it is possible to scale a given frame   so that the resulting frame is Parseval.

 Following   \cite{CC}  and \cite{KOPT}, we say that a  frame $\B=\{v_1, ...,\, v_N\} $ is {\it scalable} if there exists positive constants $\ell_1$, ..., $\ell_N$ such that
$\{\ell_1 v_1$,..., $\ell_N v_N\}$ is a Parseval frame. 
Two  Parseval N-frames    are   {\it equivalent} if one can be transformed into the other with a rotation of coordinates and the reflection   of one or more vectors.   
 A frame is {\it nontrivial} if no two vectors are parallel. In the rest of the paper,  when we say  "unique" we will always mean    "unique up to an equivalence", and we will often assume without saying that frames are nontrivial. 

 It is well known that Parseval  n-frames   are orthonormal (see also  Corollary \ref{C-orthon}). Consequently,
for  given unit vector $w$,  
 there is a  unique Parseval n-frame  that contains $w$.  If $||w||\ne 1$,  no Parseval  n-frame  contains $w$.

When $N>n$ and $||w||\leq 1$, there are infinitely many non-equivalent Parseval $N$-frames that contain $w$ 
\footnote{ We are indebted to P. Casazza for this remark.}.
By the main theorem in \cite{CL},    it is possible to construct a Parseval frame  $ \{v_1, ...,\, v_N\}$ with vectors of prescribed lengths  $0<  \ell_1 $, ..., $\ell_N \leq 1 $  that satisfy  $\ell_1^1+...+\ell_N^2=n$. 
We can let $\ell_N=||w||$ and,  after  a  rotation of coordinates,  assume that $v_N=w$, thus proving that  the Parseval frames that contain $w$ are as many as the
sets of  constants $\ell_1$, ... $\ell_{N-1}$.   

But when $N=n+1$, there is a class of   Parseval frames  that can be uniquely constructed from a given vector: precisely, all {\it triangular} frames, that is, frames
  $\{v_1, \, ... ,\, v_{N}\} $  such that the matrix $(v_1,\, ... \, v_n)$ whose columns are $v_1$, ..., $v_n$ is right triangular.    
We recall that a matrix $\{a_{i,j}\}_{1\leq i,\, j\leq n }$ is {\it right-triangular} if $a_{i,j}=0 $ if $i>j$.

The following theorem  will be proved in Section 3.

 \begin{Thm}\label{T-Uniq-n+1}  Let $\B=\{v_1,\, ...,\, v_n,\, w\}$  be a triangular Parseval  frame, with 
 $||w||<1$.  Then $\B$ is unique, in the sense that if $\B'=\{v_1',\, ...,\, v_n',\, w\}$   is  another triangular Parseval  frame, then $v_j'=\pm v_j$.
\end{Thm}

Every frame is equivalent, through a rotation of coordinates  $\rho$, to a  triangular frame, 
and so  Theorem \ref{T-Uniq-n+1} implies that every Parseval $(n+1)$-frame  that contains a given vector $w$   is equivalent to one which is uniquely determined  by $\rho(w)$. 
However, that does not imply that the frame itself is uniquely determined by $w$  because the rotation $\rho$ depends also on the other vectors of the frame.

%

\medskip
We  also study the problem of determining whether a
given frame  $\B =\{v_1, ...,v_n, v_{n+1}\} \subset\R^n$  is scalable or not. Assume   $||v_j||=1$, and  let $ \theta_{i,j}\in [0,\pi) $ be the angle  between $v_i$ and $v_j$.

If
  $\B$   contains an orthonormal basis, then the problem has no solution, so we assume that this not the case. 
 We prove the following 
   
\begin{Thm}\label{T-NandSuff}   $\B$ is scalable  if and only  there exist constants $\ell_1$, ...,\, $\ell_{n+1}$ such that  for every $i\ne j$
\begin{equation}\label{e-id-cosine}
(1-\ell_i^2)(1-\ell_j^2)= \ell_i^2 \ell_j^2\cos\theta_{i,j}^2.
\end{equation}
\end{Thm}

The identity \eqref{e-id-cosine}  has several interesting consequences (see corollary \ref{C-boundsforlength}). First of all, it shows that     $ \ell_j^2\leq 1$; if 
 $\ell_i=1$ for some $i$,   we also have  $ \cos\theta_{i,j}^2=0$    for every $j$, and so 
$v_i$ is orthogonal to all  other vectors.  This interesting fact  is also true for  other  Parseval  frames,  and is a consequence of the following  

\begin{Thm}\label{TNec-Parframe}
Let $\B=\{v_1, ...,\, v_{N  }\}$  be a Parseval frame.  Let $||v_j||=\ell_j$. Then 
\begin{equation}\label{e-necN}
\sum_{j= i}^{N}  \ell_j^2 \cos^2 \theta_{ij}  = 1 
\quad 
\sum_{j= i}^{N} \ell_j^2 \sin^2 \theta_{ij} = n-1
\end{equation}
  \end{Thm}
  
  The identities   \eqref{e-necN} are probably known, but we did not find a reference in the literature.
It is worthwhile to remark that  from \eqref{e-necN} follows that  $$
\sum_{j= i}^{N} \ell_j^2 \cos^2 \theta_{ij}  -
\sum_{j= i}^{N}  \ell_j^2 \sin^2 \theta_{ij}=\sum_{j= i}^{N}  \ell_j^2 \cos(2 \theta_{ij})  = 2-n.$$  When  $n=2$, this identity is proved in   Proposition \ref{PCasazza-n=2}. 

%
%

Another consequence of Theorem \ref{T-Uniq-n+1} is the following
\begin{Cor}\label{C-nec-cosines}
If $\B=\{v_1, ..., v_n, v_{n+1}\}$ is a scalable frame, then, for every $1\leq i\leq n+1$, and every $  j\ne k\ne i$  and $k'\ne j'\ne i$,
\begin{equation}\label{e-2cos-ratios}
\frac{|\cos\theta_{k,j}|}{|\cos\theta_{k,j}|+|\cos\theta_{k, i}\cos\theta_{j,i}|}= \frac{|\cos\theta_{k',j'}|}{|\cos\theta_{k',j'}|+|\cos\theta_{k', i}\cos\theta_{j',i}|}.
\end{equation}
\end{Cor}
\medskip

We prove  Theorems \ref{T-Uniq-n+1}, \ref{T-NandSuff} and \ref{TNec-Parframe}  and their corollaries in Section   3. In Section 2 we prove some preliminary results and lemmas.

\medskip
\noindent
{\it Acknowledgement.} We wish to thank Prof. P. Casazza and Dr.  J. Cahill for stimulating conversations.

\section{Preliminaries}


We refer to \cite{CK} or to  \cite{HKLW} for the definitions and basic properties of finite frames. 

We recall that    ${\mathcal B} =\{ v_1, ... ,\, v_N \} $   is a Parseval  frames in $\R^n$ if and only if   rows of the matrix $(v_1, ..., v_N)$ are orthonormal.  
Consequently,   
 $ 
 \sum_{i=1}^N ||v_i||^2=n.
 $ 
   If the   vectors in $\B$ have all    the same length, then   $||v_i||=\sqrt {n/N}$. See e.g. \cite{CK}.

We will often let  $\vec  e_1= (1, 0,\, ...,\,0)$, ... $\vec  e_n= (0,   \, ...,0,\,\,1)$, and we will denote by 
$(v_1, ..., \hat v_k, ... v_N)$ the matrix with the column  $v_k$ removed.

 To the best of our knowledge, the following proposition is due to P.  Casazza (unpublished-2000)  but  can also be found  in  \cite{GK} and in the recent preprint \cite{CKLMNPS}.
 
    \begin{Prop}\label{PCasazza-n=2}   $\B=\{v_1,...,\,v_N\} \subset \R^2$ 
   is a tight frame if and only   if for some index $i\leq N$,  
\begin{equation}\label{e-idn=2} \sum_{j=1}^N ||v_j|| ^2 e^{2i\theta_{i,j}}=0 .\end{equation}
\end{Prop}

It is easy to verify that if \eqref{e-idn=2} is valid for some index $i$, then it is valid for all other  $i$'s.

\begin{proof}
 %
 Let $\ell_j=||v_j||$. After a rotation, we can let $ v_i =v_1= (\ell_1,0)$ and $\theta_{1, j}=\theta_j$, so that  $v_j= (\ell_j\cos\theta_{ j}, \ell_j\sin\theta_{ j})$.

 
 $\B$ is a tight frame with frame constant $A$ if and only if  the rows of the matrix $(v_1, ..., v_N)$ are  orthogonal and have length $A$. That implies 
  \begin{equation}\label{eq1}
  \sum_{j=1}^N \ell_j^2  \cos^2\theta_{ j}=\sum_{j=1}^N \ell_j^2 \sin^2\theta_{ j}=A
  \end{equation}
  and
   \begin{equation}\label{eq2}
  \sum_{j=1}^N \ell_j^2 \cos  \theta_{ j} \sin\theta_{ j}=0.  
  \end{equation}
 From \eqref{eq1} follows that $\sum_{j=1}^N \ell_j^2(\cos^2\theta_{ j}-\sin^2\theta_{ j})= \sum_{j=1}^N \ell_j^2  \cos(2\theta_{ j})=0$, and from \eqref{eq2}   that $ \sum_{j=1}^N \ell_j^2  \sin(2\theta_{ j})=0$, and so we have proved \eqref{e-idn=2}.
 
If  \eqref{e-idn=2} holds,  then \eqref{eq1} and \eqref{eq2}  hold  as well, and from these identities follows that $\B$ is a tight frame.

 \end{proof}
 
 \medskip

\begin{Cor}\label{C3vectorsR2}
 Let  $\B=\{v_1,\ v_2, v_3\}\subset \R^2$ be a tight frame. Assume that  the  $v_i$'s have all the same length. Then,    
  $  \theta_{1,2}  =\pi/3$, and   
  $  \theta_{1,3} =2\pi/3$.
  \end{Cor}
  
So, every  such  frame $\B$  is equivalent to a   dilation of the   "Mercedes-Benz frame"   $ \left\{ (1, 0), \ (-\frac{1}{2}, \frac{\sqrt 3}{2}),  (-\frac{1}{2}, -\frac{\sqrt 3}2 )\right\}.$ 

\begin{proof} Let  $v_1=(1,0)$, and  $\theta_{1, i}= \theta_i$ for simplicity.  By Proposition \ref{PCasazza-n=2}, $1+\cos(2\theta_2)+ \cos(2\theta_3)=0$, and $\sin(2\theta_2)+ \sin(2\theta_3)=0$.  It is easy to verify that  these equations are satisfied only when $\theta_2=\frac \pi 3$ and  $  \theta_{ 3} =\frac{2\pi}3$ or viceversa.\end{proof}

\medskip

The following simple proposition  is a special case of  Theorem \ref{T-Uniq-n+1}, and will be a necessary step in the proof. 

\begin{Lemma}\label{L-3vectR2}
Let $w=(\alpha_1,\,\alpha_2)$ be   given.  Assume $||w||  <1$.  There exists a unique nontrivial Parseval  frame 
$\{v_1, v_2, w\}\subset \R^2$,   with  $v_1= (a_{11}, 0)$,  $v_2= (a_{1,2},\  a_{2,2})$, and $a_{1,1}$, $a_{2,2}>0$.
\end{Lemma}

\begin{proof}  We  find $ a_{1,1}$, $  a_{1,2}$ and $ a_{2,2} $  so that the rows of the matrix \newline $  \left(\begin{matrix}  & a_{11}, &a_{12}, &\alpha_1  \\  &0, &a_{22},  & \alpha_2 \end{matrix}\right)$ are orthonormal.  That is, 
\begin{equation}\label{e-tec1}
\alpha_1^2+ a_{1,1}^2+ a_{1,2}^2=1, \quad \alpha_2^2+a_{2,2} ^2=1, \quad \alpha_1\alpha_2+ a_{1,2}a_{2,2}=0.
\end{equation}
From the second equation, $a_{2,2} = \pm\sqrt{1-\alpha_2^2}  $;  if we can chose $a_{2,2}>0$,  from the  third equation we obtain 
$a_{1,2}=-\frac{\alpha_1\alpha_2}{\sqrt{1-\alpha_2^2}}   $ and from the first equation
$$ a_{1,1}^2 =1-\alpha_1^2-a_{12} ^2 =  1-\alpha_1^2- \frac{\alpha_1^2\alpha_2^2}{ 1-\alpha_2^2}=\frac{1-\alpha_1^2-\alpha_2^2  }{1-\alpha_2^2}.  $$  
 Note that $a_{1,1}^2>0$   because   $||w||^2= \alpha_1^2+\alpha_2^2<1$. We can chose then $$a_{1,1}=   \frac{\sqrt{1-\alpha_2^2-\alpha_1^2}} {\sqrt{1-\alpha_2^2} }.$$
Note also that $v$ and $v_2$ cannot be parallel; otherwise,  $\frac { a_{1,2} }{a_{2,2}}= -\frac{\alpha_1\alpha_2}{1-\alpha_2^2} =\frac {\alpha_1}{\alpha_2}\iff -\alpha_2^2=1-\alpha_2^2$, which is not possible.
 \end{proof}

\noindent{\it Remark.} The proof   shows   that $v_1$ and $v_2$ are uniquely determined by $w$. It shows also that if   $||w||=1$,  then $a_{1,1}=0$, and consequently  $v_1=0$.

 \section{Proofs}
 
In this section we prove Theorem \ref{T-Uniq-n+1} and some of its corollaries.

\begin{proof}[ Proof of Theorem \ref{T-Uniq-n+1} ]   Let  $   w= (\alpha_1, ..., \alpha_n)$.
 We construct a  nontrivial Parseval frame ${\cal M}=\{v_1, ..., v_n, w\}\subset \R^n$ with the following properties:   the matrix  $(v_1, ...,v_n) =\{a_{i,j}\}_{1\leq i,j\leq n}$    is right  triangular, and 
 \begin{equation}\label{e-ajj}
 a_{j,j}=\begin{cases} \sqrt{1-\alpha_n^2} & \mbox{ \ if\ }  j=n  \cr  \frac{\sqrt{1- \sum_{k=j }^{n}\alpha_{k}^2}}{\sqrt{1-\sum_{k=j+1}^{n }\alpha_{k}^2 }} & \mbox{ \ if\ } 1\leq j<n. \end{cases}
 \end{equation}
   
  The proof will show that ${\cal M}$ is unique, and also that the assumption that $||w||<1$ is necessary in the proof.

    To construct the vectors $v_j$   we argue
by induction on $n$.  When $n=2$ we have already proved the result in  Lemma \ref{L-3vectR2}.  We now assume  that    the lemma is valid in dimension  $n-1$, and we show that it is valid also in dimension $n$. 

 Let  $ \t w= (\alpha_2, ..., \alpha_n)$. By assumptions, there exist vectors $\t v_2, ...,\, \t v_n $ such that 
   the set  $\widetilde {\cal M}=  \{\t v_2, ..., \t v_n, \t w\}\subset \R^{n-1}$
   is a  Parseval frame, and the matrix $(\t v_2, ..., \t v_n, \t w)$ is right triangular and invertible.  If we assume that the elements of the diagonal are positive, 
   the   $\t v_j$'s are uniquely determined by  $w$.  We let  $\t v_j= (a_{2,j}, ..., a_{n,j})$, with $a_{k, j}=0$ if $k<j$ and $a_{j,j}>0$.

   We  show that  $\widetilde {\cal M}$ is the projection  on $\R^{n-1}$ of a Parseval frame in $\R^n=\R\times\R^{n-1}$ that satisfies the assumption of the theorem. To   this aim, we prove that there exist   scalars  $x_1$,..., $ x_n$ so that the vectors $\{v_1, ..., v_{n+1}\}$  which are defined by
    \begin{equation}\label{e-def-vj}
    v_1= (x_1, 0, ..., 0), \quad v_j= (x_j, \t v_j)\ \mbox{ if} \ 2\leq j\leq n,\    v_{n+1}=w 
\end{equation}      
form a Parseval  frame of $\R^n$.  
   The proof is in various steps: 
 first, we construct   a unit vector $(y_2, ..., y_{n+1})$ which is orthogonal to the rows of the matrix $(\t v_2, ... , \t v_n, \t w)$. Then, we show  that there exists   $-1<\lambda<1$ so that $\lambda y_{n+1}=    \alpha_1$.
  Finally, we chose $x_1= \sqrt{1-\lambda^2}$, $x_j=\lambda y_j$, and we prove that the vectors $ v_1, ..., v_{n+1} $ defined in \eqref{e-def-vj}  form a Parseval frame that satisfies the assumption of the lemma.  
    
    \medskip
    
  First of all, we observe that $\{ v_1, ..., v_{n+1}\} $ is a Parseval frame if and only if   $\vec x=(x_1, x_2, ..., x_n, \alpha_1)$  is a unit vector that  satisfies 
  the orthogonality conditions:
 %
\begin{equation}\label{1} (\t v_2, ... \t v_n,\t w)   \vec x = \left(\begin{matrix}  a_{22}   & a_{23} &...   &a_{2, n} &\alpha_2\\ 
 0    &a_{33} &...    &a_{3, n} & \alpha_3
\\ 
  \vdots &\vdots  &\vdots  &\vdots&\vdots
\\
 0     & 0 & ...    & a_{n,n} &\alpha_n
\end{matrix}\right)    \left(\begin{matrix}     x_2\\        \vdots \\  x_n  \\  \alpha_1\end{matrix}\right)\!\! = \!\left(\begin{matrix}  0\\  0 \\   \vdots \\  0 \end{matrix}\right).
 \end{equation}
   By a  well known formula of linear algebra,  
   the vector  
   \begin{equation}\label{e-formula-det}  \vec y= y_2\vec e_2+...+ \vec e_{n+1} y_{n+1}  
   =
   \mbox{det}\left(\begin{matrix}    \vec e_2   &\vec e_3  &...   &\vec e_{n } &\vec e_{n+1}\\   a_{2, 2}  &a_{2, 3}  &...   &a_{2, n} &\alpha_2\\ 
 0  &a_{3, 3}  &...   &a_{3, n} & \alpha_3
\\ 
  \vdots &\vdots  &\vdots  &\vdots & \vdots
\\
  0  &0  &...    & a_{n-1, n}&\alpha_{n-1}
\\
 0  &0  &...   & a_{n,n} &\alpha_n
\end{matrix}\right)  
\end{equation}
 is   orthogonal to the rows of the matrix in \eqref{1}, and so it is a constant multiple of $\vec x  $.    That is, 
 $\vec x= \lambda \vec y $ for some $\lambda \in\R$.
 
 \medskip
Let us  prove   that $|| \vec y||=1$. The  rows of the matrix $(\t v_2, ..., \t v_n, \t w)$ are orthonormal, and so after a rotation       \begin{equation}\label{prod}  (\t v_2, ..., \t v_n, \t w) = 
\left(
 \begin{matrix}    &0  &1  &...   &0 &0\\ 
&0  &0 &  ...   &0 & 0
\\ 
& \vdots &\vdots     &\vdots&\vdots&\vdots
\\
&0  &0  &...   & 1&0
\\
&0  &0  &...   & 0 &1\end{matrix}\right).\end{equation}
The formula  in \eqref{e-formula-det} applied with the matrix in  \eqref{prod} produces the vector $\vec e_1=(1, 0, ...,0)$. Thus,  $\vec y$ in \eqref{e-formula-det} is a rotation of $\vec e_1$, and so it is a unit vector as well.

 We now prove that   $|\lambda| <1$. From $ \vec x= (x_2, ..., x_n, \alpha_1) =\lambda (y_2, ... , y_n,y_{n+1})$,  we obtain $\lambda=\alpha_1/y_{n+1} .$  
 By \eqref{e-formula-det}, $$   y_{n+1} = (-1)^{n+1} \mbox{det} \left(\begin{matrix}     &a_{2, 2}  &a_{2, 3}  &... & a_{2, n-1} &a_{2, n}  \\ 
&0  &a_{3, 3}  &...  & a_{3, n-1}&a_{3, n}  
\\ 
& \vdots &\vdots  &\vdots  &\vdots &\vdots
\\
&0  &0  &...  &  a_{n-1,n-1} & a_{n-1, n} 
\\
&0  &0  &...  & 0& a_{n,n}  
\end{matrix}\right)   = (-1)^{n+1} \prod_{j=2}^n a_{j,j}.
$$ 
Recalling that by \eqref{e-ajj},  $a_{j,j}= \frac{\sqrt{1- \sum_{k=j }^{n}\alpha_{k}^2}}{\sqrt{1-\sum_{k=j+1 }^{n }\alpha_{k}^2 }}$, 
 we can see at once that 

 %
 \begin{align} \nonumber y_{n+1} &= (-1)^{n+1}\prod_{j=2}^n a_{j,j}  = (-1)^{n+1}\sqrt{1- \alpha_2^2-...-\alpha_{n-1}^2-\alpha_n^2} \\\label{e-prod-ajj} = & (-1)^{n+1}\sqrt{1-||w||^2+\alpha_1^2}.\end{align}
 In view of  $\lambda y_{n+1}=\alpha_1$, we obtain 
$$\lambda= (-1)^{n+1}\frac{\alpha_1}{\sqrt{1- ||w||^2+\alpha_1^2 }} . $$ 
 Clearly, $|\lambda|<1$ because $||w||<1$.  We now let   \begin{equation}\label{e-a11} x_1 = \sqrt{1-\lambda^2} = \frac{\sqrt{1-||w||^2}}{\sqrt{1- ||w||^2+\alpha_1^2}},\end{equation} and we define the $v_j$'s as in \eqref{e-def-vj}. The first  rows of the matrix  $(v_1, ...,\, v_{n+1})$ is $(\sqrt{1-\lambda^2},\vec x)= (\sqrt{1-\lambda^2},\lambda \vec y)$, and so it is unitary and  perpendicular to the other rows. Therefore, the $\{v_j\}$ form a tight frame that satisfies the assumption of the theorem.  
%
\end{proof}

\medskip
The proof of Theorem  \ref{T-Uniq-n+1}   shows the following interesting fact:    
 By \eqref{e-prod-ajj} and \eqref{e-a11}
  $$ 
  \mbox {det} (v_1, ..., v_n) =  \prod_{j=1}^n a_{jj} = x_1 \prod_{j=2}^n a_{jj}= \sqrt{1-||w||^2}.
$$
  This formula   does not depend on the fact that 
  $(v_1, ...,v_n)$ is right triangular,  because every $n\times n$ matrix  can be reduced in this form with a rotation that does not alter its determinant and does not alter  the norm of $w$. This observation proves the following   

\begin{Cor} \label{C-determinant}
Let $\{w_1, ..., w_{n+1}\}$ be  a Parseval frame.  Then,  $$\mbox {det} ( w_1, ...,\hat w_j, ..., w_{n+1})= \pm \sqrt{1-||w_j||^2}.
$$
\end{Cor}

  \medskip
  \begin{proof}[Proof of Theorem \ref{T-NandSuff}]

  Let $ \ell_j=||v_j||$. If $\{v_1, ..., v_{n+1}\}$ is a Parseval frame, then the rows of the matrix  $B=(   v_1 , ...,   v_{n+1} )$  are orthonormal. 
While proving  Theorem \ref{T-Uniq-n+1}, we have constructed a vector $\vec x= (x_1, ..., x_{n+1})$, with 
$x_j=   (-1)^{j+1} \lambda $ det$( v_1, ...,\hat v_j, ..., v_{n+1})$, which is  perpendicular to the rows of $B$. 
By  Corollary \ref{C-determinant}, 
$x_j=\pm \sqrt{1-\ell_{j}^2}$.  Since    $\B$ is a Parseval frame,   $ \ell_1^2+...+ \ell_{n+1}^2=n$, and so 
$$
||\vec x||^2=  x_1^2+... +x_{n+1}^2= (1-\ell_1^2)+ ... + (1-\ell_{n+1}^2)=1.
$$
So, the $(n+1)\times (n+1)$ matrix $\t B$  which is obtained from $B$ with the addition of the row $\vec y$, is 
unitary, and therefore also the columns of $\t B$ are orthonormal. 
For every $  i, j\leq n+1$,
\begin{equation}\label{id1}
\langle   v_i,  v_j\rangle \pm \sqrt{1-\ell_{i}^2 }\sqrt{1-\ell_{j}^2}= \ell_i \ell_j\cos\theta_{ij}\pm 
\sqrt{1-\ell_i^2}\sqrt{1-\ell_j^2}=0
\end{equation}
which implies \eqref{e-id-cosine}.

\medskip
\noindent
Conversely, suppose that \eqref{e-id-cosine} holds.   By \eqref{id1}, the vectors $\t v_j=(\pm \sqrt{1-\ell_j^2}, \ v_j)$ are orthonormal for some choice of the sign $\pm$; therefore,
 the  columns of the matrix $\t B$ are orthonormal, and so also the rows are orthonormal, and $\B$ is a Parseval frame.

\end{proof}

\medskip
\begin{Cor}\label{C-boundsforlength}
Let $\B=\{v_1, ... v_{n+1}\}$ be a nontrivial Parseval frame. Then, $ \frac 1{n+1} <\ell_j^2 $ for every $j$. Moreover,  for all $j$ with the possible exception of one, 
  $\frac 12 <\ell_j^2 $.
\end{Cor}

\begin{proof} The identity  \eqref{e-id-cosine}  implies that, for $i\ne j$,  
\begin{equation}\label{44a} 1-\ell_j^2-\ell_i^2 +\ell_i^2\ell_j^2\sin^2\theta_{ij}=0.\end{equation} 
That implies  $\ell_j^2+\ell_i^2 \ge  1$ for every $i\ne j$, and so all $\ell_j^2$'s, with the possible exception of one, are $\ge \frac 12$.
 Recalling that $\sum_{i=1}^{n+1}\ell_i^2=n$,   
$$1- (n+1)\ell_j^2 + \sum_{i=1}^{n+1} \ell_j^2\ell_i^2 \sin^2\theta_{ij} =0, 
$$
and so   $\ell_j^2>\frac{1}{n+1}$. 
 %
 \end{proof}
 
 \medskip
\begin{proof}[Proof of Theorem \ref{TNec-Parframe}]   After a rotation, we can assume $v_i=v_1= (\ell_1, 0, ... ,0)$. We let $\theta_{1,j}=\theta_j$  for simplicity.
With this rotation  $v_j= (\ell_j\cos\theta_{ j} , \ \ell_j\sin\theta_{ j}  w_j)$ where $   w_j$ is a unitary vector in $\R^{n-1}$.
The rows of the matrix  $(v_1, ...   v_{N})$  are orthonormal, and so the norm of the first row is  
\begin{equation}\label{33}
\sum_{j\ge 1}\ell_j^2\cos^2\theta_{ j}+ \ell_1^2=1
\end{equation}
The   projections of $v_2$, .... $v_{N}$ over a hyperplane that is orthogonal to $v_1$  form a tight frame on this hyperplane. That is to say that   $\{\ell_2\sin\theta_{  2} w_2, \ ... , \ \ell_{N}\sin\theta_{  N} w_{N}\}$ is a tight frame in $\R^{n-1}$, and  so it satisfies 
\begin{align}\nonumber
&\ell_2^2\sin\theta_{ 2}^2 ||w_2||_2^2+ ... + \ell_{N}^2\sin\theta_{  N}^2 || w_{N}||_2^2 \\ \label{44}&= 
\ell_2^2\sin\theta_{ 2}^2 + ... + \ell_{N}^2\sin\theta_{  N}^2 =n-1.
\end{align}
\end{proof}

\begin{Cor}\label{C-orthon}
$\{v_1,v_2,...,v_n\}$ is a Parseval frame in $\R^n$ if and only if  the $v_i$'s are  orthonormal
\end{Cor}

\begin{proof}  
By   \eqref{33}, all vectors in a Parseval frame have length $\leq 1$. By \eqref{44}
$$
\sum_{j=1}^{n} \ell_j^2 \sin^2 \theta_{i,j} = n-1   
$$
which implies that $\ell_j=1$ and $\sin\theta_{ij}=1$ for every $j\ne i$, and so all vectors are orthonormal. 
\end{proof}
 
\begin{proof}[Proof of Corollary \ref{C-nec-cosines}]
Assume that $\B$ is scalable; fix $i<n+1$, and  chose $j\ne k\ne i$. By  \ref{e-id-cosine} 
\begin{align*}
(1-\ell_i^2)(1-\ell_j^2)&= \ell_i^2 \ell_j^2\cos\theta_{i,j}^2,\\  (1-\ell_i^2)(1-\ell_k^2)&= \ell_i^2 \ell_k^2\cos\theta_{i,k}^2,
\\
(1-\ell_k^2)(1-\ell_j^2)&= \ell_k^2 \ell_j^2\cos\theta_{k,j}^2.
\end{align*}
These equations are easily solvable for  for $\ell_1^2,\ \ell_j^2$ and $\ell_k^2$; we obtain
$$\ell_i^2= \frac{|\cos\theta_{k,j}|}{|\cos\theta_{k,j}|+|\cos\theta_{k, i}\cos\theta_{j,i}|}.$$ 
This expression for $\ell_i$ must be independent of the choice of $j$ and $k$, and so \eqref{e-2cos-ratios} is proved
\end{proof}

\end{document}